\documentclass{patmorin}
\usepackage{amsthm,amsmath,graphicx,stmaryrd,amsopn,url}
\usepackage{pat}
\newcommand{\n}{N}

\title{\MakeUppercase{Crossings in Grid Drawings}}

\author{%
  Vida Dujmovi\'c,\thanks{Department of Mathematics and Statistics
       and Department of Systems and Computer Engineering, 
       Carleton University}\,\, 
  Pat Morin,\thanks{School of Computer Science, Carleton University}\,\, and 
  Adam Sheffer\thanks{School of Computer Science, Tel Aviv University}}

\DeclareMathOperator{\crs}{cr}

\DeclareMathOperator{\skp}{skip}
\DeclareMathOperator{\ncs}{ncs}
\DeclareMathOperator{\volume}{vol}

\begin{document}
\begin{titlepage}
\maketitle

\begin{abstract}
  We prove tight crossing number inequalities for geometric graphs whose
  vertex sets are taken from a $d$-dimensional grid of volume $\n$
  and give applications of these inequalities to counting the number
  of crossing-free geometric graphs that can be drawn on such grids.

  In particular, we show that any geometric graph with $m\geq 8N$
  edges and with vertices on a 3D integer grid of volume $N$, has
  $\Omega((m^2/n)\log(m/n))$ crossings. In $d$-dimensions, with $d\ge 4$,
  this bound becomes $\Omega(m^2/n)$. We provide matching upper bounds
  for all $d$. Finally, for $d\ge 4$ the upper bound implies that the
  maximum number of crossing-free geometric graphs with vertices on some
  $d$-dimensional grid of volume $N$ is $n^{\Theta(n)}$. In 3 dimensions
  it remains open to improve the trivial bounds, namely, the $2^{\Omega(n)}$
  lower bound and the $n^{O(n)}$ upper bound.
\end{abstract}

\end{titlepage}


\section{Introduction}

The study of crossings in drawings of graphs has a long history.
\emph{Euler's Formula} states that the maximum number of edges in an $n$
vertex \emph{planar graph}---one that can be drawn in the plane without
crossings---is $3n-6$.  Using Euler's Formula and careful counting, Ajtai
\etal~\cite{ajtai.chvatal.ea:crossing-free} showed that any plane drawing
of a graph with $n$ vertices and $m\ge 4n$ edges has at least $c m^3/n^2$
crossing pairs of edges, for some constant $c\ge 1/100$.  The same authors
used this to prove their main result: The maximum number of planar graphs
that can be embedded on any fixed set of $n$ points is $2^{O(n)}$.

The lower bound, $cm^3/n^2$, on the number of crossings in a plane
drawing has since become known as ``the \emph{Crossing Lemma}''
or ``the Crossing Number Inequality'' and has subsequently found many
other applications.  Sz\'ekely~\cite{szekely:crossing} showed that
this inequality can be used to give very simple proofs of many results
in incidence geometry, including a proof of the Szemer\'edi-Trotter
Theorem on point-line incidences~\cite{szemeredi.trotter:extremal}.
Sz\'ekely's method has since been used for many combinatorial geometry
problems; the most famous of these applications is probably the result
of Dey~\cite{dey:improved} on the maximum number $k$-sets of a point set.

Ajtai \etal's proof of the Crossing Lemma uses the probabilistic method
in the sense of Chv\'atal~\cite{chvatal:hypergraphs}: The proof works by
summing the number of crossings in two different ways.  More recently, a
``from the book'' proof of the Crossing Lemma that uses a more literal
application of the probabilistic method to obtain a better constant,
$c\ge 1/64$, was discovered by Chazelle, Sharir, and Welzl (See Aigner
and Ziegler~\cite[Chapter~30, Theorem~4]{aigner.ziegler:proofs}).  Pushing
this argument even further, Pach \etal~\cite{pach.radoicic.ea:improving}
currently hold the record for the largest constant, $c\ge 1/33.75$.

The main result of Ajtai \etal---that the maximum number
of crossing-free graphs that can be drawn on any point set of
size $n$ is $2^{O(n)}$---has also been the starting point
for many research problems. The original bound, which was
$O(10^{13n})$, has been improved repeatedly to the current record of
$O(187.53^n)$~\cite{sharir.sheffer:counting*1}.  The result has also
been tightened for special classes of crossing-free graphs including
triangulations ($O(30^n)$)~\cite{sharir.sheffer:counting}, spanning
cycles ($O(54.55^n)$) \cite{sharir.sheffer.ea:counting}, perfect
matchings ($O(10.05^n)$)~\cite{sharir.welzl:on}, spanning trees
($O(141.07^n)$)~\cite{hoffmann.sharir.ea:counting}, and cycle free graphs
($O(160.55^n)$)~\cite{hoffmann.sharir.ea:counting,sharir.sheffer:counting}.
A webpage containing an up-to-date compendium of these types of results
is maintained by the third author \cite{sheffer:numbers}.

\subsection{Geometric Grid Graphs}

Thus motivated by the importance of Ajtai \etal's results, the goal
of the present paper is to extend their results to graph drawings in
higher dimensions.  In particular, we extend their results to graphs
drawn on grids.  For any positive integers $X_1,\ldots,X_d$, the
$d$-dimensional \emph{$X_1\times\cdots\times X_d$ grid} is a finite
subset of the $d$-dimensional natural lattice, $\N^d$, given by
\[  \N(X_1,\ldots,X_d) = \{(x_1,\ldots,x_d): 
      \mbox{$x_i\in\{1,\ldots,X_i\}$ for all $i\in\{1,\ldots,d\}$}\}
	\enspace .\]
The \emph{volume} of the $X_1\times\cdots\times X_d$ grid is
$\prod_{i=1}^d X_i$, i.e,. the number of points in the grid.

A ($d$-D) \emph{geometric (grid) graph}, $G$, is a graph with vertex
set $V(G)\subseteq \N^d$.  Throughout this paper, for two vertices $u$
and $w$ in a geometric graph, $G$, we will use the notation $uw$ to
refer both to the open line segment with endpoints $u$ and $w$ and to
the edge $uw\in E(G)$, if present.    The \emph{volume}, $\volume(G)$,
of $G$ is the volume of the minimal $X_1\times\cdots\times X_d$ grid
that contains $V(G)$.

A geometric grid graph, $G$, is \emph{proper} if, for every edge $uw\in
E(G)$, and every vertex $x\in V(G)$, we have that $x\not\in uw$.  That is,
$G$ is proper if no edge passes through a vertex.  For the remainder of
this paper, all geometric grid graphs we refer to are proper.  From this
point onwards, the phrase ``geometric grid graph'' should be interpreted
as ``proper geometric grid graph.''

Two edges $uw$ and $xy$ in a geometric grid graph \emph{cross} if they have a
point in common.  When this happens, we say that $uw$ and $xy$ form
a \emph{crossing}.  We define $\crs(G)$ as the number of crossings
in $G$.\footnote{Note that this is different from the planar crossing
number, usually also denoted $\mathrm{cr}(G)$, that is the minimum number of
crossings in any drawing of the (non-geometric) graph $G$.}  We say that
$G$ is \emph{crossing-free} if $\crs(G)=0$.  Finally, we define
\[ \crs_d(\n,m)=\min\left\{\crs(G):\mbox{%
    $G$ is a $d$-D geometric grid graph, $|E(G)|=m$, and $\volume(G)\le\n$}
   \right\} \enspace .
\]
That is, $\crs_d(\n,m)$ is the minimum number of crossings in any
$d$-D geometric grid graph with $m$ edges and volume no more
than $\n$.

We are also interested in the maximum number, $\ncs_d(\n)$, of crossing-free
$d$-D geometric grid graphs that can be drawn on any particular grid of
volume at most $\n$.  That is,
\[
  \ncs_d(\n) = \max\left\{
     \left|\left\{G:\mbox{$V(G)\subseteq\N(X_1,\ldots,X_d)$
            and $\crs(G)=0$} \right\}\right| :
    \mbox{$\prod_{i=1}^d X_i \le\n$} \right\} \enspace .
\]

Results on plane drawings of graphs have immediate implications
for $\crs_2(\n,m)$ and $\ncs_2(\n)$:
\begin{enumerate}
  \item  Euler's Formula implies that $\crs_2(\n,3\n-5)\ge 1$,
  \item  Ajtai \etal's Crossing Lemma implies that $\crs_2(\n,m)\ge
  cm^3/\n^2$ for $m\ge 4\n$, and
  \item  Ajtai \etal's upper-bound of $2^{O(n)}$ on the number of planar
  graphs that can be drawn on any planar point set of size $n$ implies that
  $\ncs_2(\n)\in 2^{O(\n)}$.
\end{enumerate}

Bose \etal~\cite{bose.czyzowicz.ea:maximum} show that the maximum number
of edges in a crossing-free $d$-D geometric grid graph of volume $\n$ is at
most $(2^d-1)\n-\Theta(\n^{(d-1)/d})$. This result is analogous to Euler's
Formula in the sense that it shows that such graphs have a linear number
of edges.  It also implies, for example, that $\crs_d(\n,(2^d-1)\n)\ge 1$.
Since Euler's Formula is the main property of planar graphs used by Ajtai
\etal\ to prove their results, it seems reasonable that bounds similar
to those of Ajtai \etal\ should hold for $d$-D geometric grid graphs.

The key difference, however, is that unlike Euler's formula, the bound
of Bose \etal\ depends on the volume, $\n$, of the grid and not on the
number, $n$, of vertices in the graph.  For $d\ge 3$ it is not possible
to obtain non-trivial bounds on crossings that depend only on the
number of edges and vertices.  For example, there exists a crossing-free
3-D geometric grid graph of volume $N$ that has $n=N^{1/3}$ vertices
and $m=\binom{n}{2}$ edges \cite{cohen.eades.ea:three-dimensional}.

\subsection{New Results}

In the current paper, we study $\crs_d(\n,m)$ and $\ncs_d(\n)$ for $d\ge
3$ and prove the results shown in \tabref{results}.  In \tabref{results},
and throughout this paper, we assume that $d$ is a constant that is
independent of $\n$ and $m$, so that the $O$, $o$, $\Omega$, $\omega$,
and $\Theta$ notations hide factors that depend only on $d$.

\begin{table}
  \begin{center}
    \begin{tabular}{r|lllc}
      $d$ & $\crs_d(\n,m)$ & $\ncs_d(\n)$ & References \\ \hline
      2 & $\Theta(m^{3}/\n^2)$ & $2^{\Theta(\n)}$ &~\cite{ajtai.chvatal.ea:crossing-free} \\
      3 & $\Theta((m^2/\n)\log(m/\n))$ &  & Theorems~\ref{thm:3d-lower-bound} and \ref{thm:3d-upper-bound} \\
      $\ge 4$ & $\Theta(m^{2}/\n)$ & $2^{\Theta(\n\log\n)}$ & Theorems~\ref{thm:4d-lower-bound}, \ref{thm:4d-upper-bound}, and \ref{thm:4d-counting} \\
    \end{tabular}
  \end{center}
  \caption{Old and new results on crossings in $d$-D geometric grid graphs.}
  \tablabel{results}
\end{table}

Our results show that the situation in three and higher dimensions
is significantly different than in two dimensions. For all $d \ge
4$, and $m\ge 2^d\n$, $\crs_d(\n,m)\in \Theta(m^2/\n)$ and even
$\crs_3(\n,m)$ is only $\Theta((m^2/\n)\log(m/\n))$.
There are therefore geometric grid graphs with $\Omega(\n^2)$ edges
that have only $O(\n^3)$ crossings ($O(\n^3\log\n)$ crossings in
3-d).  In contrast, in 2 dimensions, any graph with $n$ vertices and
$\Omega(n^2)$ edges has $\Omega(n^4)$ crossings.

For $d\ge 4$, the bounds on $\crs_d(\n,m)$ are strong enough to
show that $\ncs_d(\n)\in 2^{\Theta(\n\log\n)}$.  Thus, the number,
$2^{\Theta(\n\log\n)}$, of crossing-free graphs whose vertex set comes
from a specific $d$-dimensional grid having $\n$ points is much larger
than the number, $2^{\Theta(\n)}$,  of crossing-free graphs that can be
drawn on any planar point set of size $\n$.

\subsection{Related Work}

The study of crossing-free 3-D geometric grid graphs is an
active area in the field of graph drawing.  A \emph{$d$-D
grid drawing} of a graph, $G$, is a
one-to-one mapping $\varphi:V(G)\rightarrow \N^d$.  Any drawing,
$\varphi$, yields a geometric grid graph, $\varphi(G)$, with
vertex set $V(\varphi(G))=\{\varphi(u):u\in V\}$ and edge set
$E(\varphi(G))=\{\varphi(u)\varphi(w):uw\in E(G)\}$.  The drawing
$\varphi$ is \emph{crossing-free} if the geometric grid graph $\varphi(G)$
is crossing-free and the \emph{volume} of $\varphi$ is the volume of
$\varphi(G)$.

Cohen \etal~\cite{cohen.eades.ea:three-dimensional} showed that the
complete graph, $K_n$, on $n$ vertices, and therefore any graph on $n$
vertices, has a crossing-free 3-D grid drawing of volume $O(n^3)$ and this
is optimal.  However, for many classes of graphs, sub-cubic volume 3-D
grid drawings are possible; this includes 
sufficiently sparse graphs ($O(m^{4/3}n)$)
  \cite{dujmovic.wood:three-dimensional}, 
  graphs with maximum degree $\Delta$ and other $\Delta$-degenerate
  graphs 
  ($O(\Delta mn)$, $O(\Delta^{15/2}m^{1/2}n$)) 
  \cite{dujmovic.wood:three-dimensional,dujmovic.wood:upward},
$\chi$-colorable graphs 
  ($O(\chi^2n^2)$, $O(\chi^6m^{2/3}n)$)
  \cite{pach.thiele.ea:three-dimensional,dujmovic.wood:three-dimensional}, 
graphs taken from some proper minor-closed family of graphs 
  ($O(n^{3/2})$)
  \cite{dujmovic.wood:three-dimensional}, 
planar graphs ($O(n\log^{16} n)$)
  \cite{battista.frati.ea:on},
outerplanar graphs ($O(n)$)
  \cite{felsner.liotta.ea:straight-line},
and graphs of constant treewidth ($O(n)$) \cite{dujmovic.morin.ea:layout}.

The work most closely related to the current work, in that it presents
an extremal result relating crossings, volume, and number of edges, is
that of Bose \etal~\cite{bose.czyzowicz.ea:maximum}, who show that the
maximum number of edges in a crossing-free $d$-D geometric grid graph,
$G$, with vertex set $V(G)\subseteq \N(X_1,\ldots,X_d)$, is exactly
\begin{equation}
    \prod_{i=1}^d (2X_i-1) - \prod_{i=1}^d X_i \enspace . \eqlabel{czyzo}
\end{equation}
For a fixed volume, $\n=\prod_{i=1}^d X_i$, maximizing \eqref{czyzo} gives
$X_1=\cdots=X_d=\n^{1/d}$, in which case \eqref{czyzo} becomes $(2^d-1)\n -
\Theta(\n^{(d-1)/d})\le (2^d-1)\n$.  We state this here as lemma since we
make use of it several times.
\begin{lem}[Bose \etal\ 2004]\lemlabel{boser}
  In any crossing-free $d$-D geometric grid graph, $G$, of volume $\n$
  $|E(G)|\le(2^d-1)\n$.
\end{lem}
\lemref{boser} immediately yields the upper-bound
$\ncs_d(\n)\in2^{O(\n\log\n)}$ (see the beginning of \secref{counting}).
It also yields the lower-bound $\crs_d(m)\ge m-(2^{d}-1)N$ since, if a
geometric grid graph $G$ has $m\ge 2^{d-1}N$ we can remove an edge an
edge from $G$ that eliminates at least one crossing.  Since this can
be repeated until $G$ has $m\le 2^{d-1}N$ edges, this implies that $G$
has at least $m-(2^d-1)N$ crossings.

Finally, we note that Bukh and Hubard \cite{bukh.hubard:space} have
defined a form of crossing number for 3-dimensional geometric graphs
that are not necessarily grid graphs.  In their definition, a 4-tuple
of vertex-disjoint edges form a \emph{space crossing} if there is
a line that intersects every edge in the 4-tuple.  The \emph{space
crossing number}, $\crs_4(G)$, of a 3-d geometric graph, $G$, is the
number of space crossings formed by $G$'s edges.  They show that a
3-d geometric graph $G$ with $n$ vertices and $m\ge 4^{41}n$ edges
has a space crossing number $\crs_4(G) \in \Omega(m^6/(n^4\log^2 n))$.
An easy lifting argument shows that this bound on the space crossing
number almost implies the Crossing Lemma; specifically, it shows that
the number of crossings in a graph with $m$ vertices and $n$ edges drawn
in the plane is $\Omega(m^3/(n^2\log n))$.

\section{3-Dimensional Geometric Grid Graphs}

In this section, we present upper and lower bounds on $\crs_3(\n,m)$.
Here, and in the remainder of the paper we use the notation
$u_i$, $i\in\{1,\ldots,d\}$, to denote the $i$th coordinate of the
$d$-dimensional point $u$.  Thus for a point $u\in\R^3$, $u_1$, $u_2$,
and $u_3$ are $u$'s x-, y-, and z-coordinates, respectively.

\subsection{The Lower Bound}
\seclabel{3d-lower-bound}

\begin{thm}\thmlabel{3d-lower-bound}
  For all $m\ge 8\n$, $\crs_3(\n,m) \in \Omega((m^2/\n)\log(m/\n))$.
\end{thm}

\begin{proof}
  Let $G$ be any geometric grid graph with $V(G)\subseteq\N(X,Y,Z)$,
  with $XYZ\le\n$, and $|E(G)|=m$. (That is, $G$ is a 3-D geometric
  grid graph with $m$ edges and volume at most $\n$.)  We may assume,
  without loss of generality, that no edge of $G$ contains any point
  of the $X\times Y\times Z$ grid in its interior; any such edge can
  be replaced with a shorter edge without introducing any additional
  crossings and without changing $m$ or $\n$. This assumption is subtle,
  but important, and is equivalent to assuming that, for every edge $uw$
  of $G$, $\gcd(u_1-w_1, u_2-w_2, u_3-w_3)=1$.

  For any integer, $p\ge 1$, define the $X\times Y\times Z$
  \emph{$p$-grid} as the set of points
  \[
    \{(x/p,y/p,z/p): x\in\{p,p+1,\ldots,pX\},\,
    y\in\{p,p+1,\ldots,pY\},\, z\in\{p,p+1,\ldots,pZ\}\} \enspace ,
  \]
  which we denote by $\N(pX,pY,pZ)/p$.  Observe that the size of the
  $p$-grid is at most $\n p^3$.

  In order to avoid double-counting later on, we need to define a
  sequence of point sets that are disjoint.  To achieve this, we define
  the \emph{essential $p$-grid} as follows: If $p$ is a prime number,
  then the essential $p$-grid is the $p$-grid minus the 1-grid.  Otherwise
  ($p$ is composite), let $p_1,\ldots,p_k$ be the primes in the prime
  factorization of $p$.  We begin with the points of the $p$-grid and
  then remove all points that are contained in the $p_i$-grid, for each
  $i\in\{1,\ldots,k\}$.  What remains is the essential $p$-grid. (See
  \figref{grid} for a 2-dimensional illustration.)  Observe that the
  essential $p$-grid and the essential $q$-grid have no points in common
  for any $p\neq q$.

  \begin{figure}
    \begin{center}
      \includegraphics{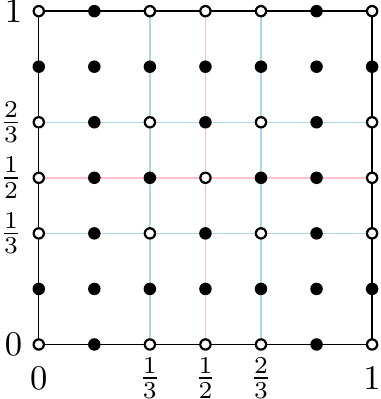}
    \end{center}
    \caption{A 2-dimensional piece of the essential 6-grid.  Removing
     the 1-grid, 2-grid, and 3-grid from the 6-grid leaves the
     essential 6-grid.}
    \figlabel{grid}
  \end{figure}

%
%

  Next, observe that each edge $uw$ of $G$
  contains the $p$-grid points
  \[
      P_{uw}^p = \{ u+(i/p)(w-u) : i\in\{1,\ldots,p-1\} \} \enspace 
  \]
  and the essential $p$-grid points
  \[
      Q_{uw}^p = \{ u+(i/p)(w-u) : i\in\{1,\ldots,p-1\}\mbox{
         and  }\gcd(i,p)=1 \} \enspace .
  \]
  The fact that $Q_{uw}^p$ contains only essential $p$-grid points follows
  from the assumption that
  $\gcd(u_1-w_1, u_2-w_2, u_3-w_3)=1$.  More specifically, the points
  in $Q_{uw}^p$ are clearly on the $p$-grid, so the only concern is
  that some of these points are on the $q$-grid for some $q<p$.  To see
  why this is not possible, observe that, if some point in $Q_{uw}^p$
  were on the $q$-grid, for some $q<p$, this would imply that
  \[  ((i/p)x,(i/p)y,(i/p)z) = (a/q,b/q,c/q) \]
  for some integers $x=w_1-u_1$, $y=w_2-u_2$, $z=w_3-u_3$, $a$, $b$, and
  $c$ such that $\gcd(i,p)=1$ and $\gcd(x,y,z)=1$. Rewriting this gives
  \begin{equation}
     (x,y,z) = (pa/(iq),pb/(iq),pc/(iq)) \enspace .   \eqlabel{snapper}
  \end{equation}
  Each value on the right hand side has a factor of $p$ in the numerator,
  so they must also have a factor of $p$ in the denominator, $iq$.
  Otherwise, $\gcd(pa/(iq),pb/(iq),pc/(iq)) > 1 = \gcd(x,y,z)$.  But this
  is not possible since $\gcd(i,p)=1$ and $q<p$, so $\gcd(iq,p)\le q < p$.

  The size of the set $Q_{uw}^p$ is a well-studied quantity and is given
  by the \emph{Euler totient function} $\varphi(p)=p\prod_{q|p}(1-1/q)$
  \cite[Section~5.5]{hardy.wright:introduction}.  Therefore, the total
  number of incidences between points of the essential $p$-grid and
  edges of $G$ is at least $m\cdot\varphi(p)$.  In understanding the
  calculations that follow, it is helpful to pretend that $\varphi(p)\ge
  cp$ for some constant $0<c<1$, though this is not strictly correct
  since, for some $p$, $\varphi(p)\in O(p/\log\log p)$.

  Let $x_1,\ldots,x_\ell$ denote the essential $p$-grid points that are
  incident on at least one edge and let $R_i$, $i\in\{1,\ldots,\ell\}$,
  denote the number of edges incident to $x_i$.  Observe, then, that
  there are $\binom{R_i}{2}$ crossing pairs of edges that cross at
  $x_i$.  From the preceding discussion, we have $\sum_{i=1}^\ell R_i
  \ge m\cdot\varphi(p)$.  Therefore, the total contribution of crossings that
  occur on the essential $p$-grid to $\crs_3(G)$ is at least
  \[
      \sum_{i=1}^\ell\binom{R_i}{2} \ge \ell \binom{m\cdot\varphi(p)/\ell}{2}
      \ge \Omega(\ell (m\cdot\varphi(p)/\ell)^2)
      = \Omega\left(\frac{m^2\varphi(p)^2}{p^3N}\right) \enspace . 
  \]
  The first inequality is an application of Jensen's Inequality
  to the function $f(x)=x(x-1)/2$.  Since $\ell\le \n p^3$,
  $\binom{m\cdot\varphi(p)/\ell}{2}\ge 0$ for $p\le \sqrt[3]{m/\n}$
  and the second inequality holds
   for $p\le \sqrt[3]{m/\n}$.

  Summing over $p$ we obtain:
  \begin{align}
     \crs(G) 
       & \ge \sum_{p=1}^{\lfloor\sqrt[3]{m/\n}\rfloor} 
                \Omega\left(\frac{m^2\varphi(p)^2}{p^3N}\right) \notag \\
       & = \Omega(m^2/\n)\sum_{p=1}^{\lfloor\sqrt[3]{m/\n}\rfloor} 
                 \varphi(p)^2/p^3 \eqlabel{a} \\
             & = \Omega(m^2/\n)\log(m/\n) \enspace . \eqlabel{b}
  \end{align}
  The step from \eqref{a} to \eqref{b}, in which we claim that
  $\sum_{i=1}^k \varphi(i)^2/i^3\in\Omega(\log k)$, is not immediate.
  To justify this step, recall the following result on Euler's totient
  function \cite[Theorem~330]{hardy.wright:introduction}:
  \[
      \sum_{i=1}^n \varphi(i) = (3/\pi^2)n^2 + O(n\log n) \enspace .
  \]
  Using this result, and the Cauchy-Schwartz Inequality, we obtain
  \[
    (3/\pi^2)n^2 + O(n\log n) = \sum_{i=1}^n \varphi(i) 
     = \sum_{i=1}^n\varphi(i)\cdot 1 \le \sqrt{\sum_{i=1}^n \varphi(i)^2}\cdot\sqrt{\sum_{i=1}^n 1}
     = \sqrt{\sum_{i=1}^n \varphi(i)^2}\cdot\sqrt{n} \enspace .
  \]
  Dividing by $\sqrt{n}$ and squaring, we obtain
  \begin{equation}
    \sum_{i=1}^n \varphi(i)^2 \ge (9/\pi^4)n^3 + O(n\log^2 n) \ge (9/\pi^4)n^3 \ge n^3/11 \enspace . \eqlabel{chimp}
  \end{equation}
  On the other hand, $\varphi(i) < i$, so
  \begin{equation}
    \sum_{i=1}^n \varphi(i)^2 
         < \sum_{i=1}^n i^2\le \sum_{i=1}^n n^2 = n^3 \enspace . \eqlabel{chomp}
  \end{equation}
  We can now justify the step from \eqref{a} to \eqref{b}
  as follows:
  \begin{align*}
    \sum_{i=1}^k \varphi(i)^2/i^3 
      &\ge \sum_{j=1}^{\lfloor\log_3 k\rfloor}
             \sum_{i=3^{j-1}+1}^{3^j} \varphi(i)^2/i^3  \\
      &\ge \sum_{j=1}^{\lfloor\log_3 k\rfloor}\frac{1}{3^{3j}}
             \left(\sum_{i=3^{j-1}+1}^{3^j} \varphi(i)^2\right)  \\
      &= \sum_{j=1}^{\lfloor\log_3 k\rfloor}\frac{1}{3^{3j}}
             \left(\sum_{i=1}^{3^j}\varphi(i)^2-\sum_{i=1}^{3^{j-1}} \varphi(i)^2\right)  \\
      &\ge \sum_{j=1}^{\lfloor\log_3 k\rfloor}\frac{1}{3^{3j}}
             \left(3^{3j}/11 - (3^{j-1})^3\right)  
              & \text{(by using \eqref{chimp} and \eqref{chomp})} \\
      &\ge \sum_{j=1}^{\lfloor\log_3 k\rfloor}\frac{1}{3^{3j}}
             \left(3^{3j}/11 - 3^{3j}/27\right)  \\
      &= \sum_{j=1}^{\lfloor\log_3 k\rfloor} \Omega(1) = \Omega(\log k) \enspace . \qedhere
  \end{align*}
\end{proof}

\subsection{The Upper Bound}
\seclabel{3d-upper-bound}

In this section, we prove the following result:

\begin{thm}\thmlabel{3d-upper-bound}
  For all $m\le\n^2/4$, $\crs_3(\n,m) \in O((m^2/\n)\log (m/\n))$.
\end{thm}

The proof of \thmref{3d-upper-bound} follows easily from the following
lemma:
\begin{lem}\lemlabel{2-layer}
  There exists a 3-D grid drawing of the complete bipartite graph
  $K_{k^2,k^2}$ on the $k\times k\times 2$ grid with $O(k^6\log k)$
  crossings.
\end{lem}

Before proving \lemref{2-layer}, we first show how it implies
\thmref{3d-upper-bound}: For simplicity, in what follows, assume
$\sqrt{\n}$, $\sqrt{m/\n}$, and $\n/\sqrt{m}$ are each integers.
Apply \lemref{2-layer}, with $k=\sqrt{m/\n}$ and tile the
$\sqrt{\n}\times\sqrt{\n}\times 2$ grid with $\n/k^2$ copies of
this drawing. The resulting geometric graph has $2\n$ vertices,
$m=\n k^4/k^2=\n k^2$ edges and
\[ 
   O((\n/k^2)k^6\log k) = O(\n k^4\log k) = O((m^2/\n)\log(m/\n))
\] 
crossings, as required by \thmref{3d-upper-bound}.

\begin{proof}[Proof of \lemref{2-layer}]
  The drawing of $K_{k^2,k^2}$ is the obvious one; each point of the
  $k\times k\times2$ grid with z-coordinate 1 is connected by an edge
  to every point with z-coordinate 2.  We denote the resulting
  geometric graph by $G_{k^2}$.

  We start by considering some edge $uw$ with $u_3=1$ (so
  $w_3=2$) and counting the number of edges that intersect $uw$.
  Let $\pi_1,\pi_2,\ldots$ be the planes that contain $uw$ and at least
  one additional vertex.  Observe that each plane $\pi_j$ contains, and
  is uniquely determined by, a line in the plane $\{z\in\R^3: z_3=1\}$
  that contains $u$ and some vertex, $x$, and such that $ux$ does not
  contain any other point of $\Z^3$ i.e.,
  $\gcd(u_1-x_1,u_2-x_2)=1$ (see \figref{planes}).
  \begin{figure}
    \centering{\includegraphics{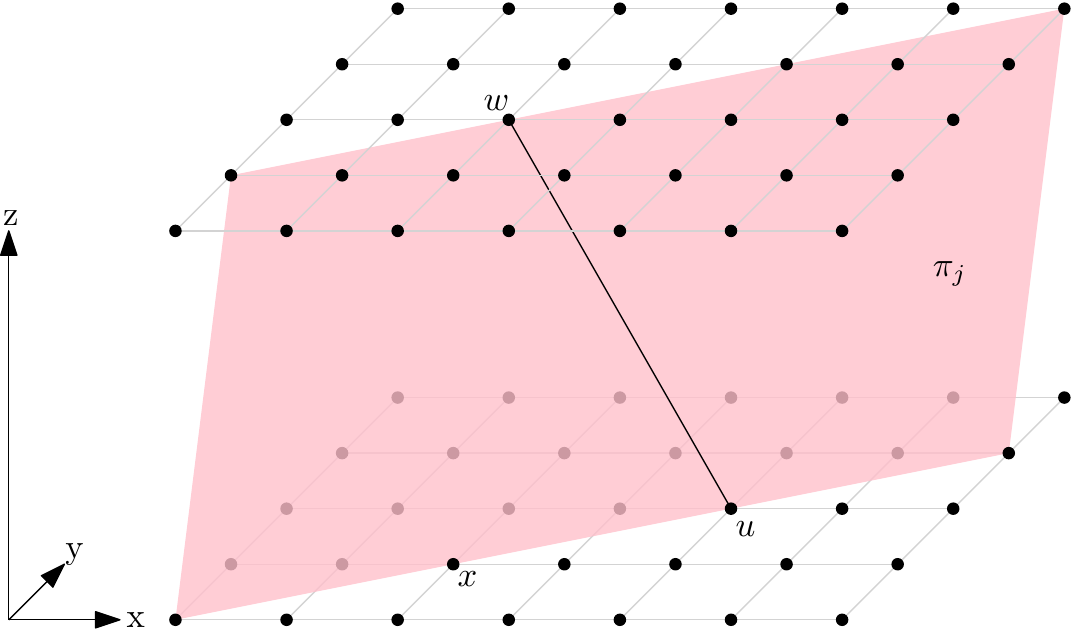}}
    \caption{A plane $\pi_j$, defined by a point $x$, with
      $\skp(\pi_j)=2$. (z-coordinates are exaggerated.)}
    \figlabel{planes}
  \end{figure}
  Define the \emph{skip} of $\pi_j$ as
  \[
     \skp(\pi_j)=\max\{|u_1-x_1|,|u_2-x_2|\} \enspace .
  \]
  Observe that, if $\skp(\pi_j)=r$, then $\pi_j$ contains at most $2k/r$
  vertices of $G$ other than $u$ and $w$ and therefore contains at most
  $(k/r)^2$ edges that cross $uw$.  Furthermore, the number of planes
  $\pi_j$ such that $\skp(\pi_j)=r$ is at most $4r$ since each such
  plane is defined by two antipodal lattice points on the boundary of a
  square of side length $2r$ centered at $u$; see \figref{skip-count}.
  Therefore, the total number of edges that cross $uw$ is at most
  \begin{equation}
     \sum_{r=1}^k 4r(k/r)^2 = 4k^2\sum_{r=1}^k 1/r 
       \le 4k^2\ln k + O(k^2) \enspace .
        \eqlabel{critical}
  \end{equation}
  Since this is true for each of the $k^4$ edges, $uw$, we conclude
  that the total number of crossings in $G_{k^2}$ is at most $2k^6\ln
  k+O(k^6)\in O(k^6\log k)$, as required.
  \begin{figure}
    \centering{\includegraphics{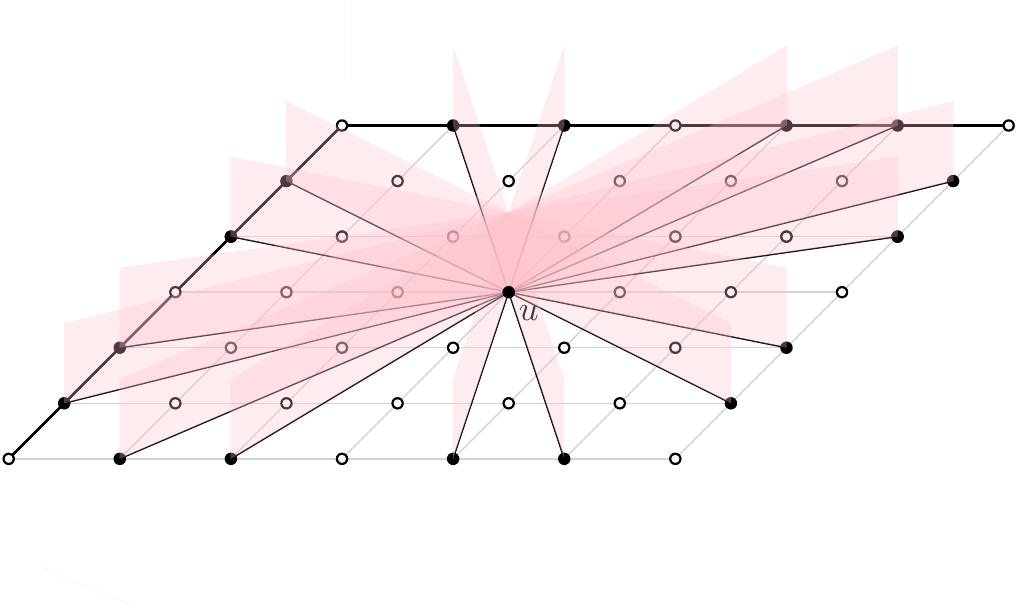}}
    \caption{Each plane $\pi_j$ with $\skp(\pi_j)=r$ is defined by two
      antipodal lattice points on the boundary of a square of side length
      $2r$ that is centered at $u$.}
    \figlabel{skip-count}
  \end{figure}
\end{proof}

\section{Higher Dimensions}

Next, we prove matching upper and lower bounds on $\crs_d(\n,m)$ for
$d\ge 4$.

\subsection{The Lower Bound}

\begin{thm}\thmlabel{4d-lower-bound}
  For all $m\ge 2^d\n$, $\crs_d(\n,m)\in\Omega(m^2/\n)$.  In particular,
  $\crs_d(\n,m)\ge \frac{1}{2}(m^2/(2^d-1\n) - m)$, for all $m\ge 0$.
\end{thm}

\begin{proof}
  Let $G$ be any geometric grid graph with $m$ edges whose vertex set
  is contained in the $X_1\times \cdots\times X_d$ grid of volume
  at most $\n$.  Note that,
  for each edge $uw\in E(G)$, the \emph{midpoint} $(u+w)/2$ of $uw$
  is contained in the $X_1\times\cdots\times X_d$ essential 2-grid:
  \[
    \N(2X_1,\ldots,2X_d)/2=\{(x_1,\ldots,x_d)/2:
         \mbox{$x_i \in \{2,3,\ldots,2X_i\}$, 
               $i\in\{1,\ldots,d\}$}\} \setminus \N(X_1,\ldots,X_d)\enspace .
  \]
  This 2-grid contains $K\le (2^d-1)\n$ points.  Let $R_i$ be the number of
  edges of $G$ whose midpoint is the $i$th point of this 2-grid. Then
  the number of crossings in $G$ is
  \begin{align*}
   \crs(G) &\ge \sum_{i:R_i\ge 1}\binom{R_i}{2} \\
    & = \frac{1}{2}\left(\sum_{i:R_i\ge 1}(R_i^2-R_i)\right) \\
    & = \frac{1}{2}\left(\sum_{i:R_i\ge 1} R_i^2 - m\right) \\
    & \ge \frac{1}{2}\left(K(m/K)^2 - m\right) \\
    & = \frac{1}{2}\left(m^2/K - m\right) \\
    & \ge \frac{1}{2}\left(\frac{m^2}{(2^{d}-1)\n} - m\right) \enspace . \qedhere
  \end{align*}
%
\end{proof}

\subsection{The Upper Bound}

\begin{thm}\thmlabel{4d-upper-bound}
  For all $d\ge 4$ and all $m\le\n^2/4$, $\crs_d(\n,m)\in O(m^2/\n)$.
\end{thm}

\begin{proof}
  Let $\ell=k^{d-1}$ for some integer $k$.  As in the proof of
  \thmref{3d-upper-bound}, it suffices to show that one can draw the
  complete bipartite graph $K_{\ell,\ell}$ on the $k\times\cdots\times
  k\times 2$ grid so that it has $O(k^{3(d-1)})=O(\ell^3)$ crossings.

  The remainder of this proof has the same structure as the proof of
  \lemref{2-layer}.  The drawing of $K_{\ell,\ell}$ we use is the graph
  $G$ with $V(G)=\N(k,\ldots,k,2)$ and
  \[
    \E(G) = \{ uw\in V(G)^2 :\mbox{$u_d = 1$ and $w_d = 2$} \} \enspace .
  \]
  
  Consider some edge $uw$ of $G$ with $u_d=1$ and $w_d=2$.  Our strategy
  is to upper bound the number of edges that cross $uw$.  Any edge $xy$
  that crosses $uw$ is contained in some plane, $\pi$, that contains $uw$
  and $xy$.  Without loss of generality, assume $x_d=1$.  Let $x'$ be some
  point of the integer lattice $\Z^d$ that is on the line containing
  $ux$ and such that $ux'$ contains no point of $\Z^d$. (That is,
  $\gcd(u_1-x'_1,\ldots,u_{d-1}-x'_{d-1}) = 1$.)

  The plane $\pi$ that contains $uw$ and $xy$ can be expressed as
  \[
     \pi = \{ u + t(w-u) + s(x'-u) : s,t\in\R\} \enspace .
  \]
  Restricted to the subspace $S_u=\{z\in\R^d: z_d=1\}$, $\pi$
  becomes a line
  \[
     L_u = \pi\cap S_u = \{ u + s(x'-u) : s\in\R\} \enspace .
  \]
  Similarly, restricted to the subspace $S_w=\{z\in\R^d: z_d=2\}$,
  $\pi$ becomes the parallel line:
  \[
     L_w = \pi\cap S_w = \{ w + s(x'-u) : s\in\R\} \enspace .
  \]
  Observe that, since there is no point on the segment $ux'$, the only
  points of the integer lattice $\Z^d$ contained in $L_u$ are obtained
  when the parameter $s$ is an integer:
  \[
     L_u\cap\Z^d = \{ u + s(x'-u) : s\in\Z\} \enspace 
  \]
  and, similarly,
  \[
     L_w\cap\Z^d = \{ w + s(x'-u) : s\in\Z\} \enspace  .
  \]
  If we define $r=\max\{|x'_i-u_i| : i\in\{1,\ldots,d-1\}\}$, then we see
  that the number of vertices of $G$, other than $u$ and $w$, intersected
  by each of $L_u$ and $L_w$ is at most $k/r$ (since $G$'s vertices are
  contained in a box whose longest side has length $k$).  In this case,
  we define the \emph{skip} of the plane $\pi$ to be $r$.
  
  Now, consider all the planes that contain $uw$ and some other vertex
  of $G$.  We wish to determine the number of such planes with skip $r$.
  Each such plane is defined by two antipodal grid points on the boundary
  of a $(d-1)$-hypercube of side length $2r$, centered at $u$, that is
  contained in the $(d-1)$-dimensional subspace $\{z\in\R^d:z_d=1\}$.
  This hypercube has $2(d-1)$ facets and each facet contains
  $(2r+1)^{d-2}$ grid points.  Therefore, the total number of planes with
  skip $r$ is at most
  \[
     (d-1)(2r+1)^{d-2} = (d-1)\left((2r)^{d-2} + O(r^{d-3})\right)
  \]
  Each plane with skip $r$ contains at most $(k/r)^2$ edges that cross
  $uw$.  Therefore, the number, $X_{uw}$, of edges that cross $uw$
  is at most
  \begin{align*}
     X_{uw} & \le \sum_{r=1}^k(d-1)\left((2r)^{d-2}+O(r^{d-3})\right)(k/r)^2 \\
         & = (d-1)2^{d-2}k^2\sum_{r=1}^k\left(r^{d-4} + O(r^{d-5})\right) \\
         & \le (d-1)2^{d-2}k^2\left(\frac{k^{d-3}}{d-3} 
             + O\left(k^{d-4}\log k\right)\right) 
               & \text{(since $d\ge 4$)} \\
         & = \frac{(d-1)2^{d-2}k^{d-1}}{d-3} + O(k^{d-2}\log k) \\
         & = O(k^{d-1}) = O(\ell).
  \end{align*}
  Since there are $\ell^2$ edges, the total number of crossings is
  therefore $O(\ell^3)$, as required.
\end{proof}

\section{The Number of Non-Crossing Graphs}
\seclabel{counting}

In this section, we show that, for dimensions $d\ge 4$, $\ncs_d(\n)\in
2^{\Theta(\n\log\n)}$, i.e., the maximum number of crossing-free graphs
that can be drawn on any grid of volume $\n$ is $2^{\Theta(\n\log\n)}$.

The upper bound follows easily from \lemref{boser} which states that
any crossing-free $d$-D geometric grid graph of volume $\n$ has at most
$(2^d-1)\n$ edges.  Therefore, any such graph corresponds to one of
the ways of choosing at most $(2^d-1)\n$ edges from among the at most
$\binom{\n}{2}$ possible edges.  Therefore, the number of such graphs is
at most
\[
  \ncs_d(\n) \le 2^{(2^d-1)\n}\binom{\binom{\n}{2}}{(2^d-1)\n} 
      \le 2^{(2^d-1)\n} (\n^{2})^{(2^d-1)\n} 
      = 2^{2(2^d-1)\n\log\n+(2^d-1)\n} 
      \in 2^{O(\n\log\n)} \enspace .
\]
The preceding argument is standard and is used, for example, by B\'arat
\etal~\cite[Lemma~4]{barat.matousek.ea:bounded-degree} for upper-bounding
the maximum number of crossing-free plane graphs with $m$ edges that can
be drawn on any particular set of $n$ points in the plane.

Next we show that the number of crossing-free geometric graphs that can
be drawn on the $\n^{1/(d-1)}\times \cdots\times\n^{1/(d-1)}\times 2$
grid is at least $2^{\Omega(\n\log\n)}$.

\begin{thm}\thmlabel{4d-counting}
  For all $d\ge 4$, $\ncs_d(\n) \in 2^{\Omega(\n\log\n)}$.
\end{thm}

\begin{proof}
  Let $G_\n$ denote the complete bipartite geometric graph described in
  the proof of \thmref{4d-upper-bound} with the value $\ell=\n/2$ (assuming,
  only for simplicity, that $(\n/2)^{1/(d-1)}$ is an integer).  For a geometric
  graph, $G$, let $\ncs(G)$ denote the number of crossing-free subgraphs
  of $G$ and define
  \[
     f(m) = \min\{ \ncs(G) : 
                \mbox{$G$ is a subgraph of $G_{\n}$ having $m$ edges} \}
     \enspace . 
  \]
  Our goal is to lower-bound $f(\n^2)$.  In order to do this, we establish
  a recurrence inequality and base cases.
  
  For our base cases, we have 
  \[ 
     f(m)\ge 1 \enspace ,
  \]
  for all $m\ge 0$, since the subgraph of $G_{\n}$ with no edges is
  crossing-free.
  
  Let $c$ be a constant so that $c\n$ is an integer and, for all
  sufficiently large $\n$ every edge of $G_\n$ intersects at most $c\n-1$
  other edges.  By the proof of \thmref{4d-upper-bound} such a constant
  $c$ exists.  Fix any subgraph, $G$, of $G_{\n}$ that has $m\ge c\n$ edges.
  From $G$, select any edge $e$. Then there are at least $f(m-1)$
  subgraphs of $G$ that do not include $e$.  Furthermore, $e$ intersects
  at most $c\n-1$ other edges of $G$, so there are at least $f(m-c\n)$
  subgraphs of $G$ that include $e$.  Therefore,
  \[  
     f(m) \ge f(m-1) + f(m-c\n) \enspace ,
  \]
  for $m\ge c\n$.  Repeatedly expanding the first term gives:
  \begin{align}
  f(m) & \ge  f(m-1) + f(m-c\n) \notag \\
         & \ge  f(m-1) + f(m-2c\n) \notag \\
         & \ge  f(m-2) + 2f(m-2c\n) \notag \\
         & \ge  f(m-3) + 3f(m-2c\n) \notag \\
         & \,\,\,\vdots  \notag \\
         & \ge  c\n\times f(m-2c\n) \enspace , \eqlabel{blech}
  \end{align}
  for $m\ge 2c\n$.
  
  For an integer $t$, we can iterate \eqref{blech} $t$ times to obtain
  \begin{equation}
     f(m) \ge (c\n)^t\times f(m-2ct\n) \enspace ,
     \eqlabel{recur-a}
  \end{equation} 
  for $m\ge 2ct\n$.  Taking $m=\n^2$, \eqref{recur-a} becomes
  \[
     f(\n^2) \ge (c\n)^t\times f(\n^2-2ct\n) \ge (c\n)^t \enspace ,
  \]
  for $t \le\n/(2c)$.  Taking $t=\lfloor \n/(2c)\rfloor$ then yields the
  desired result:
  \[
     f(\n^2) \ge (c\n)^{\lfloor\n/(2c)\rfloor} 
            \ge (c\n)^{\n/(2c) - 1} 
            = 2^{(\frac{\n}{2c}-1)(\log\n+\log c)}
            \in 2^{\Omega(\n\log\n)} \enspace . \qedhere
  \]
\end{proof}

We remark that the proof of \thmref{4d-counting} also works to lower-bound
the number of crossing-free matchings in $G_\n$.  When one selects an edge
$uw$ to be part of the matching one has to discard the at most $c\n$
edges of $G_\n$ that intersect $uw$ as well as the $2\n-1$ edges that
have $u$ or $w$ as an endpoint.  Thus, one discards at most $(c+2)\n$
edges and the remainder of the proof goes through unmodified.

\begin{cor}\corlabel{matchings}
  For all $d\ge 4$ and $\n=2k^{d-1}$, the number of crossing-free matchings
  with vertex set $\N(k,\ldots,k,2)$ is $2^{\Omega(\n\log\n)}$.
\end{cor}

From \corref{matchings}, we can derive a lower-bound on the number of
crossing-free spanning trees of the $k\times\cdots\times k\times 2$ grid:
\begin{cor}\corlabel{spanning-trees}
  For all $d\ge 4$ and $\n=2k^{d-1}$.  The number of crossing-free trees with
  vertex set $\N(k,\ldots,k,2)$ is $2^{\Omega(\n\log\n)}$.
\end{cor}

\begin{proof}
  Each of the crossing-free matchings counted by \corref{matchings} uses
  only edges $uw$ of $G_\n$ such that $u_d=1$ and $w_d=2$.  Each such
  matching, $M$, can be augmented into a crossing-free connected graph,
  $G_M$, with vertex set $\N(k,\ldots,k,2)$ by, for example, adding all
  edges in the set
  \[
     \{ uw :\mbox{$u,w\in \N(k,\ldots,k,2)$, $u_d=w_d$ and $\|u-w\|=1$} \} \enspace .  
  \] 
  The graph $G_M$ can be reduced to a tree, $T_M$, that includes all edges
  of $M$ by repeatedly finding a cycle, $C$, and removing any edge of $C$
  that is not part of $M$.  (An edge of $C\setminus M$ exists because $M$
  is a matching, and hence acyclic.)  After each such modification, $G_M$
  remains connected and has fewer cycles.  This processes terminates
  when $G_M$ becomes the desired tree, $T_M$.

  Thus, for each of the $2^{\Omega(\n\log\n)}$ matchings, $M$, there
  exists a spanning tree $T_M$ that contains $M$.  Any spanning tree
  with $\n$ vertices contains no more than $2^{\n-1}$ matchings and therefore,
  there are at least $2^{\Omega(\n\log\n)}/2^{\n-1}\in 2^{\Omega(\n\log\n)}$
  crossing-free spanning trees with vertex set $\N(k,\ldots,k,2)$.
\end{proof}

We finish this section by observing that our lower bounds are not
just for ``flat'' grids like the $k\times\cdots\times k\times 2$ grid.
They hold also for the ``square'' $k\times\cdots\times k$ grid.

\begin{cor}
  For all $d\ge 4$ and $\n=k^d$, the number of crossing-free matchings
  and spanning trees with vertex set $\N(k,\ldots,k)$ is
  $2^{\Omega(\n\log \n)}$
\end{cor}

\begin{proof}
  Observe that the $k\times\cdots\times k$ grid is made up of $k$ layers,
  each of which is a $k\times\cdots\times k\times1$ grid. Between any
  consecutive pair of these layers there are, by \corref{matchings},
  $2^{\Omega(k^{d-1}\log k)}$ crossing-free matchings that contain only
  edges that span both layers.  Since there are $k-1$ consecutive pairs
  of layers, there are therefore
  \[
     \left(2^{\Omega(k^{d-1}\log k)}\right)^{k-1} = 2^{\Omega(k^d\log k)} 
        = 2^{\Omega(\n\log\n)}
  \]
  crossing-frees graphs whose vertex set is the $k\times\cdots\times k$ grid.

  Note that the graphs we obtain in the preceding manner contain no
  cycles.  Therefore, to obtain a lower-bound of $2^{\Omega(\n\log
  \n)}$ on the number of spanning trees we can augment any of these
  graphs into a crossing-free spanning tree as is done in the proof of
  \corref{spanning-trees}.

  To obtain a lower-bound on the number of matchings we can simply count
  the matchings that only include edges from layer $i$ to layer $i+1$
  with $i\equiv 1\pmod 2$.  There are
  \[
     \left(2^{\Omega(k^{d-1}\log k)}\right)^{\lfloor (k-1)/2\rfloor} 
       = 2^{\Omega(k^d\log k)}
       = 2^{\Omega(\n\log\n)}
  \]
  such matchings.
\end{proof}

\section{Summary and Remarks}
\seclabel{summary}

We have given matching upper and lower bounds on the minimum number of
crossings in $d$-D geometric grid graphs with $m$ edges and volume at
most $\n$.  The upper-bound $\crs_d(\n,m)\in O(m^2/\n)$, for $d\ge 4$,
allows the application of a recursive counting technique to show the
lower-bound $\ncs_d(\n)\in 2^{\Omega(\n\log\n)}$; this is similar to way
in which Ajtai \etal\ used the lower-bound $\crs(n,m)\in\Omega(m^3/n^2)$
to show that that the maximum number of planar graphs that can be drawn
on any point set of size $n$ is $2^{O(n)}$.  This $2^{\Omega(\n\log\n)}$
lower-bound also holds if we restrict the graphs to be spanning trees
or matchings, but we know very little about spanning cycles:

\begin{op}
  Determine the maximum number of crossing-free spanning cycles whose
  vertex set is a grid of volume $\n$.
\end{op}

In what appears to be a remarkably unfortunate coincidence, the tight
bound $\crs_3(\n,m)\in \Theta((m^2/\n)\log (m/\n))$ does not allow for the
application of a recursive counting technique to determine any non-trivial
bound on $\ncs_3(\n)$.  If the upper-bound were slightly stronger, say
\[
   \crs_3(\n,m)\in O((m^2/\n)\log^{1-\varepsilon}(m/\n)) \enspace ,
\]
then this would be sufficient to prove that
$\ncs_3(\n)\in2^{\Omega(\n\log^{\varepsilon}\n)}$.  In contrast, if the
lower-bound were slightly stronger, say
\[
  \crs_3(\n,m)\in \Omega((m^2/\n)\log^{1+\varepsilon}(m/\n)) \enspace ,
\]
then this would be sufficient to prove that
$\ncs_3(\n)\in2^{O(\n\log^{1-\varepsilon}\n)}$.

\begin{op}[Wood]\oplabel{non-trivial}
  Find non-trivial bounds---$2^{o(\n\log\n)}$ or $2^{\omega(\n)}$---on
  $\ncs_3(\n)$.
\end{op}

\opref{non-trivial} was communicated to the first author by David R.\
Wood.  His motivation for asking this question comes from a question of
Pach \etal~\cite{pach.thiele.ea:three-dimensional}, who ask ``Does every
graph with $n$ vertices and maximum degree three have a crossing-free
3-D grid drawing of volume $O(n)$?''  This question remains unresolved,
even when the maximum degree three condition is relaxed to maximum
degree $O(1)$.

If \opref{non-trivial} can be answered with a  non-trivial upper bound,
then this would settle Pach \etal's question.  The number of labelled
graphs with $n$ vertices and having maximum degree $3$ is $2^{(3/2)n\log
n - O(n)}$~\cite[Appendix~A]{barat.matousek.ea:bounded-degree}. On the
other hand, if one can show that $\ncs_3(\n)\in 2^{o(\n\log\n)}$, then
for every constant $c>0$,
\[
   n!\ncs_3(cn) = n!2^{o(n\log n)} \le 2^{n\log n + o(n\log n)}
   < 2^{(3/2)n\log n - O(n)} \enspace ,
\]
for sufficiently large $n$. This would answer Pach \etal's question
in the negative; there are more labelled $n$-vertex graphs of
maximum degree three than there are labelled 3-D geometric grid
graphs of volume $cn$ for any constant $c$.   This type of counting
argument has been used successfully to answer similar questions
about geometric thickness \cite{barat.matousek.ea:bounded-degree},
distinct distances \cite{carmi.dujmovic.ea:distinct}, slope number
\cite{pach.palvolgyi:bounded}, book thickness \cite{malitz:graphs},
and queue number \cite{malitz:graphs,wood:bounded}.

Another approach to resolving the question of Pach \etal\ is to consider
that there are maximum degree 3 graphs that have some properties that
would seem to rule out a linear volume embedding.  An obvious candidate
property is that of being an \emph{expander}: There exist graphs, $G$,
with maximum degree $O(1)$ and such that, for any subset $S\subseteq
V(G)$, $|S|\le n/2$ the number of vertices of $V(G)\setminus S$ adjacent
to at least one vertex in $S$ is at least $\epsilon|S|$, for some constant
$\epsilon > 0$.

Expanders have no separator of size $o(n)$ and are therefore
non-planar~\cite{lipton.tarjan:separator}.  However, very recently
Bourgain and Yehudayoff~\cite{bourgain.yehudayoff:monotone} have
shown that there exist bounded degree graphs that are expanders and
that have constant \emph{queue number}. Through a result of Dujmovi\'c
\etal~\cite[Theorem~8]{dujmovic.por.ea:track}, this implies that there
are constant degree expanders that can be drawn on a 3-dimensional grid
with volume $O(n)$. Thus, the property of expansion is not sufficient
to rule out linear volume 3-D grid drawings.  We are still no closer
to solving Pach \etal's 14 year old problem:

\begin{op}[Pach \etal\ 1999]
  Does every graph with $n$ vertices and maximum degree three have a
  crossing-free 3-D grid drawing of volume $O(n)$?
\end{op}

\section*{Acknowledgements}

This research was initiated at the AMS Mathematics Research Communities
Workshop on Discrete and Computational Geometry, June 10--16, 2012.
The work of Vida~Dujmovi\'c and Pat Morin was partly funded by NSERC.
Adam Sheffer was partially supported by Grant 338/09 from the Israel
Science Fund and by the Israeli Centers of Research Excellence program
(Center No.~4/11).

\bibliographystyle{plainurl}
\bibliography{3dcrossings}

\end{document}